\documentclass[reqno,11pt,a4paper,fleqn]{amsart}
\usepackage{sansmath}
\sansmath
\usepackage{etoolbox}
\patchcmd{\section}{\scshape}{\sffamily\bfseries}{}{}    
\patchcmd{\subsection}{\bfseries}{\sffamily\bfseries}{}{} 
\usepackage{amsthm}
\newtheoremstyle{sansbold}
  {\topsep}          
  {\topsep}          
  {\itshape}         
  {}                 
  {\sffamily\bfseries}
  {.}                
  {.5em}             
  {}                 
\theoremstyle{sansbold}
\newtheorem{theorem}{Theorem}[section]
\newtheorem{lemma}[theorem]{Lemma}
\newtheorem{proposition}[theorem]{Proposition}
\newtheorem{corollary}[theorem]{Corollary}
\theoremstyle{definition} 

\newtheorem{remark}[theorem]{Remark}

\usepackage[T1]{fontenc}
\usepackage{amsmath,amssymb,comment}
\usepackage[colorlinks=true,linkcolor=blue,citecolor=blue,urlcolor=blue]{hyperref}
\usepackage[hyperpageref]{backref}
\newcommand{\R}{\mathbb R}
\newcommand{\F}{\mathcal F}
\newcommand{\dd}{\,dx}
\newcommand{\ds}{\displaystyle}
\DeclareMathOperator{\diver}{div}

\DeclareMathOperator{\diam}{diam}
\DeclareMathOperator{\supp}{supp}

\newcommand{\tocless}[2]{%
  \bgroup
  \let\addcontentsline=\nocontentsline
  #1{#2}%
  \egroup
}
\newcommand{\nocontentsline}[3]{}
\numberwithin{equation}{section}
 
\title[Comparing nonlinear eigenvalues]{An optimal inequality between nonlinear eigenvalues in convex sets}
\author{Francesco Della Pietra}
\address{Dipartimento di Matematica e Applicazioni ``R. Caccioppoli'', Universit\`a degli Studi di Napoli Federico II, Via Cintia, Complesso Universitario Monte S. Angelo, 80126 Napoli, Italy}
\email{f.dellapietra@unina.it}
\date{}
\begin{document}
\begin{abstract}
Let $\lambda_p(\Omega)$ be the first Dirichlet eigenvalue of the $p$-Laplacian in a bounded convex open set $\Omega\subset\R^N$, $N\ge2$. We prove that $p\mapsto\lambda_p(\Omega)^{1/p}/\pi_p$ is strictly decreasing, where $\pi_p^p$ is the first eigenvalue on the unit interval. Consequently, for $1<q<p<\infty$, the sharp upper bound for $\lambda_p(\Omega)^{1/p}/\lambda_q(\Omega)^{1/q}$ is $\pi_p/\pi_q$. This bound is not attained and is approached by slabs, resolving the conjecture on the supremum formulated by Parini (J. Convex Anal. 2017). The proof combines log-concavity of the first eigenfunction with a vector field constructed from the one-dimensional eigenfunctions. The endpoint limits recover the Hersch--Protter and reverse Cheeger (or Buser) inequalities.
\end{abstract}
\maketitle
\begin{center}
\begin{minipage}{.7\textwidth}
\tableofcontents
\end{minipage}
\end{center}
\section{Introduction}
For a bounded open set $\Omega\subset\R^N$ and $1<p<\infty$, let
\begin{equation}\label{eq:rayleigh}
\lambda_p(\Omega)=\min_{\varphi\in W^{1,p}_0(\Omega)\setminus\{0\}}
\frac{\ds\int_\Omega|\nabla\varphi|^p\dd}{\ds\int_\Omega|\varphi|^p\dd}
\end{equation}
be the optimal constant in the Poincar\'e inequality. A nonnegative minimizer $u$ is a weak solution of
\[
-\Delta_pu=\lambda_p(\Omega)u^{p-1},\qquad u\in W^{1,p}_0(\Omega),
\]
where $\Delta_pu=\diver(|\nabla u|^{p-2}\nabla u)$; it is well known that $u\in C^{1,\alpha}(\Omega)$ \cite{Tol84}, and $u>0$ in $\Omega$ by the Harnack inequality. On intervals,
\[
\lambda_p\big(0,\ell\big)=\left(\frac{\pi_p}{\ell}\right)^p,
\]
where
\begin{equation}
\label{eq:pi}
\pi_p=2(p-1)^{\frac1p}\int_0^1\frac{1}{(1-t^p)^{\frac1p}}\,dt
=\frac{2\pi(p-1)^{\frac1p}}{p\sin(\frac\pi p)}.
\end{equation}
Let us consider the quantities
\[
\Lambda_p(\Omega)=\frac{\lambda_p(\Omega)^{1/p}}{\pi_p},
\qquad \F_{p,q}(\Omega)=\frac{\lambda_p(\Omega)^{1/p}}{\lambda_q(\Omega)^{1/q}}.
\]
Then, on the interval $(0,\ell)$, $\Lambda_p=1/\ell$ for every $p$, and $\F_{p,q}=\pi_p/\pi_q$. Moreover, 
\[
\Lambda_p(t\Omega)=t^{-1}\Lambda_{p}(\Omega),\qquad \F_{p,q}(t\Omega)=\F_{p,q}(\Omega),\qquad t>0.
\]
Our aim is to study $\F_{p,q}(\Omega)$, in the context of shape optimization, when $\Omega$ is a convex domain.
Parini \cite{Par17} conjectured  that, among planar convex sets, $\F_{p,q}$ admits a minimizer and is bounded from above, that its supremum equals the interval value $\pi_p/\pi_q$ and is not attained, and that every sequence of sets with fixed area and diverging diameter is maximizing. The statement on the supremum was also formulated in every dimension and listed among the open problems in \cite{BBP23}; the general case was still open even in dimension two. Our main result is the following.
\medskip

\begin{theorem}\label{thm:main}
Let $\Omega\subset\R^N$, $N\ge2$, be a bounded convex open set. For every $1<q<p<\infty$,
\begin{equation}\label{eq:main}
\F_{p,q}(\Omega) < \frac{\pi_{p}}{\pi_{q}},
\end{equation}
that is, $p\mapsto\Lambda_p(\Omega)$ is strictly decreasing on $(1,\infty)$.
The constant in \eqref{eq:main} is sharp, and it is asymptotically achieved by slabs: if $\Omega_\ell=(0,1)\times(-\ell/2,\ell/2)^{N-1}$, then
\[
\lim_{\ell\to\infty}\F_{p,q}(\Omega_\ell)=\frac{\pi_p}{\pi_q}.
\]
\end{theorem}
\medskip

Since \eqref{eq:main} is strict and it becomes an equality in the limit along slabs, the supremum of $\F_{p,q}$ among bounded convex sets equals $\pi_p/\pi_q$ and it is not attained. Hence Theorem \ref{thm:main} proves, in every dimension, the conjecture on the supremum formulated by Parini; the endpoint exponents $q=1$ and $p=\infty$ are discussed below. We observe that the inequality $\sup\F_{p,q}\ge\pi_p/\pi_q$ among convex sets was already known from \cite[Lemma 2.4]{BBP23}. Moreover, as we will see in Corollary \ref{cor:plane} below, the statement of \cite[Conjecture 6.1]{Par17} on maximizing sequences in the plane, together with its converse, is true.

\medskip

From the point of view of shape optimization, $\F_{p,q}$ has been studied in several cases. The classical Cheeger, Hersch--Protter and reverse Cheeger inequalities are bounds for $\F_{p,q}$ at the endpoint exponents $q=1$ or $p=\infty$. To put Theorem~\ref{thm:main} in context, we recall some known results.


\subsection*{Generalized Cheeger inequalities}
Denote by $\rho(\Omega)$ the inradius of $\Omega$, that is, the radius of the largest ball contained in $\Omega$, and by
\[
h(\Omega)=\inf\left\{\frac{P(E)}{|E|}:E\subset\Omega,\ |E|>0\right\}
\]
the Cheeger constant, where $P(E)$ is the perimeter in $\R^N$. For Lipschitz domains, by \cite{KF03,JLM99} together with $\pi_p\to2$ at both endpoints, it holds
\begin{equation}\label{eq:limits}
\lim_{p\to1^{+}}\Lambda_p(\Omega)=\frac{h(\Omega)}2,
\qquad
\lim_{p\to\infty}\Lambda_p(\Omega)=\frac1{2\rho(\Omega)}.
\end{equation}
Accordingly, we use the conventions
\begin{equation}\label{eq:conventions}
\lambda_1(\Omega)=h(\Omega),\qquad
\lambda_\infty(\Omega)^{1/\infty}=\rho(\Omega)^{-1},
\qquad \pi_1=\pi_\infty=2.
\end{equation}
For general domains, it holds \cite{Lin93}
\begin{equation}\label{eq:lind}
\F_{p,q}(\Omega)\ge\frac qp,\qquad 1<q<p<\infty.
\end{equation}
Letting $q\to1$ one recovers the Cheeger inequality \cite{Che70,KF03}
\[
\lambda_p(\Omega)\ge\left(\frac{h(\Omega)}{p}\right)^p.
\]
Thus $\F_{p,q}$ is a generalized Cheeger ratio, and it is natural to look for its optimal bounds in a given class of sets.

\subsection*{The class of all open sets}
A systematic study of $\F_{p,q}$ for every $1\le q<p\le\infty$ was carried out by Briani, Buttazzo and Prinari \cite{BBP23}. In the class of all open sets of $\R^N$ with finite measure, they proved that the infimum of $\F_{p,q}$ is nonincreasing with respect to the dimension $N$ and converges to $q/p$ as $N\to\infty$, so that \eqref{eq:lind} is asymptotically sharp in high dimension \cite[Theorem 2.6]{BBP23}. Moreover, the supremum of $\F_{p,q}$ is finite if and only if $q>N$ \cite[Theorem 2.9]{BBP23}. For $q>N$ this follows from the lower bound $\lambda_q(\Omega)^{1/q}\ge C_{N,q}/\rho(\Omega)$ \cite[Theorem 2.7]{BBP23} (see also \cite{Pol15}), combined with $\lambda_p(\Omega)^{1/p}\le\lambda_p(B_1)^{1/p}/\rho(\Omega)$; for $q\le N$, removing from a set suitable compact sets with zero $q$-capacity and positive $p$-capacity leaves $\lambda_q$ unchanged and makes $\F_{p,q}$ arbitrarily large.

More recently, Bucur, Buttazzo and de Villeroch\'e \cite{BBdV26} proved that, among all open subsets of $\R^N$, both the infimum and the supremum of $\F_{p,q}$ are attained (the supremum being $+\infty$ when $q\le N$), by open sets which may be unbounded and which, when $N<q<p\le\infty$, can be chosen as complements of discrete sets; they also study the problem with a box constraint.

\subsection*{The class of convex sets}
In convex sets the picture is different. Two classical sharp inequalities hold: the first one is the Hersch--Protter inequality
\begin{equation}\label{eq:hersch}
\Lambda_p(\Omega)>\frac1{2\rho(\Omega)},
\end{equation}
proved by Hersch \cite{H60} for planar convex sets and by Protter \cite{Pr81} in any dimension when $p=2$ (see also \cite{dp26} for a recent sharper version), and extended to every $p$ in \cite{Kaj15,Br18,DPdBG20}. The second one is the reverse Cheeger (or Buser) inequality
\begin{equation}\label{eq:buser}
\Lambda_p(\Omega)<\frac{h(\Omega)}2,
\end{equation}
observed by Parini \cite{Par17} for $N=p=2$ as a consequence of a classical inequality of P\'olya \cite{Pol60}, and proved in \cite{Br20} for every $N$ and $p$; see also \cite{DPdBG20}, where it is derived from the P\'olya-type estimate of \cite{DPG14}. Both inequalities are asymptotically sharp on thinning slabs. As observed in the proof of \cite[Proposition 3.2]{BBP23}, they give
\[
\frac{\pi_p}{\pi_q}\,\frac1{h(\Omega)\rho(\Omega)}\le \F_{p,q}(\Omega)\le \frac{\pi_p}{\pi_q}\,h(\Omega)\rho(\Omega),\qquad 1<q<p<\infty,
\]
where $1\le h(\Omega)\rho(\Omega)\le N$. Hence $\F_{p,q}$ is bounded from above among convex sets, differently from what happens among all open sets when $q\le N$. As observed in \cite{BBP23}, \eqref{eq:buser} and \eqref{eq:hersch} also show that the conjecture on the supremum holds in the endpoint cases $q=1$ and $p=\infty$. 

In \cite{BBP23} it is also shown that, along convex sets with unit
volume and diameter tending to infinity, the limit points of $\F_{p,q}$
lie between the infimum and the supremum of $\F_{p,q}$ among convex sets
of dimension $N-1$ \cite[Proposition 3.5]{BBP23}. Together with
Blaschke's selection theorem, this suggests a dichotomy
\cite[Remark 3.9]{BBP23}: either there is a maximizing sequence of
unit-volume convex sets with diverging diameters, and then the supremum
in dimension $N$ equals the one in dimension $N-1$ (in the plane this
gives $\sup\F_{p,q}=\pi_p/\pi_q$, since $\F_{p,q}=\pi_p/\pi_q$ on
intervals); or a maximizing sequence stays, up to translations, in a
compact set, and then a convex maximizer exists. For the endpoint ratio $\lambda_2/h^2$, Cont, Leonardi and Saracco recently proved that sequences of convex bodies in $\R^3$ with vanishing relative width, and rhomboid-like sets with vanishing relative width in any dimension, are maximizing \cite[Theorems 5.1 and 5.6]{CLS26}.

By \eqref{eq:limits}, Theorem \ref{thm:main} bridges the reverse Cheeger inequality \eqref{eq:buser} and the Hersch--Protter inequality \eqref{eq:hersch}, which it contains, with strict inequalities, as limit
cases. With the conventions \eqref{eq:conventions}, its conclusions hold for $1\le q<p\le\infty$: the endpoint cases reduce to \eqref{eq:buser}, \eqref{eq:hersch} and $h(\Omega)\rho(\Omega)>1$ (Remark \ref{rem:endpoints}). In particular, in the dichotomy of \cite[Remark 3.9]{BBP23} the first alternative always occurs, and no convex maximizer exists. As we already mentioned, in the plane maximizing sequences are characterized by the divergence of the diameters (Corollary \ref{cor:plane}); in higher dimension, divergence of the diameter is not sufficient (Remark \ref{rem:cylinders}), while the sufficient conditions of \cite{CLS26} extend to every $\F_{p,q}$ (see Section \ref{sec:further}).

As regards minimization of $\F_{p,q}$, Parini \cite{Par17} proved the existence of a minimizer of $\lambda_2/h^2$ among planar convex sets, and his numerical computations suggest that the square is optimal; another proof, together with a sufficient criterion for existence in $\R^N$, was given by Ftouhi \cite{Ft21}. Convex minimizers of $\F_{p,q}$ exist whenever the infimum in dimension $N$ is strictly smaller than in dimension $N-1$, in particular in the plane when $q\le2\le p$ \cite[Theorems 3.6 and 3.8]{BBP23}; for $q=1$ their existence in every dimension was proved by Pratelli and Saracco \cite{PS25}. Among triangles, the optimality of the equilateral triangle for $\lambda_2/h^2$ has been recently proved by Endo, Liu and Mariano \cite{ELM26}, by means of a computer-assisted proof. To the best of our knowledge, apart from these results the minimization problem for $\F_{p,q}$ among convex sets remains open, and even for $\lambda_2/h^2$ the optimality of the square is still a conjecture.

\subsection*{Idea of the proof of \eqref{eq:main}}
The proof uses a lower bound for $\lambda_q$ obtained from a vector field, in the spirit of Protter and Hersch (see \cite[\S\S4--5]{H60}). If
\[
\diver V-(q-1)|V|^{q/(q-1)}\ge\mu
\]
in a suitable weak sense, integration by parts and the Young inequality imply $\lambda_q\ge\mu$. Starting from the positive $p$-eigenfunction $u$, we construct
\[
V=\Psi(|\nabla\log u|)\,\nu,\qquad \nu=-\frac{\nabla u}{|\nabla u|},
\]
where $\Psi$ is built from the one-dimensional eigenfunctions of the
$p$- and $q$-Laplacians. The divergence of $V$ equals the desired bound
plus a remainder, which is nonnegative for two reasons: the level sets
of the log-concave function $u$ are convex, and a one-dimensional
comparison between the two eigenfunctions fixes the sign of the
remaining factor. The critical set of $u$ is avoided by integrating
over $\{u<t\}$.
\subsection*{Plan of the paper}
In Sections \ref{sec:1d} and \ref{sec:eigen} we study the one-dimensional problem and recall the properties of the first eigenfunction on convex domains. Section \ref{sec:field} contains a vector field lower bound for $\lambda_q$. Inequality \eqref{eq:main} is proved in Section \ref{sec:proof}, and its sharpness, together with some consequences, in Section \ref{sec:sharp}. Section \ref{sec:further} contains some further questions.

\section{The one-dimensional problem}
\label{sec:1d}
In this section we give some properties of the first eigenfunction in the one dimensional case, which will be useful in the sequel. 
\medskip

\noindent For $r>1$, let $y_r$ be the positive first eigenfunction on $(0,1)$ normalized by $y_r(1/2)=1$. It is symmetric about $1/2$, strictly increasing on $(0,1/2)$, and satisfies
\begin{equation}\label{eq:profile}
-(|y_r'|^{r-2}y_r')'=\pi_r^r y_r^{r-1},\qquad y_r(0)=y_r(1)=0.
\end{equation}
For $x\in (0,1/2)$, it holds 
\begin{equation}\label{eq:energy}
(r-1)y_r'(x)^r+\pi_r^r y_r(x)^r=\pi_r^r.
\end{equation}
Separation of variables recovers \eqref{eq:pi} and gives
\begin{equation}\label{eq:inverse}
\dfrac{\displaystyle\int_0^{y_r(x)}(1-t^r)^{-1/r}\,dt}
{\displaystyle\int_0^1(1-t^r)^{-1/r}\,dt}=2x.
\end{equation}

\begin{lemma}\label{lem:profiles}
For $1<q<p<\infty$ and $0<x<1/2$,
\[
y_q(x)>y_p(x).
\]
\end{lemma}
\begin{proof}
Write 
\[
f_r(t):=\frac{1}{(1-t^r)^{\frac1r}},\;\; c_r=\int_0^1f_r(t)\,dt,\;\; \text{and}\;\;
\Theta_r(s)=\frac{1}{c_r}\int_0^s f_r(t)\,dt.
\]
 Hence \eqref{eq:inverse} can be written as
\begin{equation}
\label{eq:inverse2}
y_{r}(x)=\Theta_{r}^{-1}(2x),\quad x\in (0,1/2).
\end{equation}
 The ratio $g=f_q/f_p$ is strictly increasing, since as $0<t<1$ and $q<p$,
\[
(\log g)'(t)=\frac{t^{q-1}}{1-t^q}-\frac{t^{p-1}}{1-t^p}
=\frac{t^{q-1}-t^{p-1}}{(1-t^q)(1-t^p)}>0.
\]
Fix $0<s<1$. Since $g$ is strictly increasing, $f_q<g(s)f_p$ on $(0,s)$ and
$f_q>g(s)f_p$ on $(s,1)$. Integrating over $(0,s)$ and $(s,1)$ respectively,
\[
c_q\Theta_q(s)<g(s)\,c_p\Theta_p(s),
\qquad
c_q\bigl(1-\Theta_q(s)\bigr)>g(s)\,c_p\bigl(1-\Theta_p(s)\bigr).
\]
By dividing, $\Theta_q(s)<\Theta_p(s)$, and then by \eqref{eq:inverse2},  $y_q(x)>y_p(x)$. Since the profiles take values in $(0,1)$ on $(0,1/2)$,
$y_q^{q-1}>y_p^{q-1}\ge y_p^{p-1}$ there. 
\end{proof}

Now let us set
\[
b_r(x):=\frac{y_r'(x)}{y_r(x)},\qquad 0<x<1/2.
\]
It is a smooth decreasing bijection from $(0,1/2)$ onto $(0,\infty)$. More precisely, by substituting $y_r'=b_ry_r$ in \eqref{eq:profile} 
\begin{equation}\label{eq:slope}
-(r-1)b_r^{r-2}b_r'=\pi_r^r+(r-1)b_r^r,
\qquad b_r(0)=\infty,\quad b_r(1/2)=0.
\end{equation}
Integrating \eqref{eq:profile} from $x$ to $1/2$ gives
\[
y_r'(x)^{r-1}=\pi_r^r\int_x^{1/2}y_r(t)^{r-1}\,dt.
\]
Combining this formula with \eqref{eq:energy}, we obtain the useful identity
\begin{equation}\label{eq:gamma}
-\frac{d}{dx}\log b_r(x)^{r-1}
=\frac{\pi_r^r}{b_r(x)^{r-1}}+(r-1)b_r(x)
=\frac1{\Gamma_r(x)}
\end{equation}
where
\[
\Gamma_r(x):=y_r(x)\int_x^{1/2}y_r(t)^{r-1}\,dt.
\]
\smallskip
\begin{lemma}\label{lem:transfer}
Let $1<q<p<\infty$ and $\lambda>0$. Set
\[
a=\frac{\lambda^{1/p}}{\pi_p},\qquad \mu=(a\pi_q)^q.
\]
Define $\Psi:(0,\infty)\to(0,\infty)$ by
\begin{equation*}\label{eq:transfer}
\Psi(a b_p(x))=(a b_q(x))^{q-1},\qquad 0<x<1/2.
\end{equation*}
Then $\Psi$ is smooth and satisfies
\begin{align}
&\frac{z^{2-p}(\lambda+(p-1)z^p)}{p-1}\Psi'(z)
-(q-1)\Psi(z)^{q/(q-1)}=\mu,\label{eq:ode}\\
&\left(\frac{\Psi(z)}{z^{p-1}}\right)'<0.\label{eq:sign}
\end{align}
\end{lemma}
\begin{proof}
The definition of $\Psi$ is unambiguous because $b_p$ is a smooth bijection with nonzero derivative. Smoothness and positivity follow from the corresponding properties of $b_p$ and $b_q$.

Put $z(x)=ab_p(x)$ and $W(x)=(ab_q(x))^{q-1}$. Equation \eqref{eq:slope} gives
\[
z'(x)=-\frac1a\frac{z^{2-p}(\lambda+(p-1)z^p)}{p-1},
\quad
W'(x)=-\frac1a\big(\mu+(q-1)W^{q/(q-1)}\big).
\]
Since $W=\Psi(z)$, division of these identities proves \eqref{eq:ode}.

To prove \eqref{eq:sign}, differentiate the logarithm of the ratio along $z(x)$. By \eqref{eq:gamma},
\begin{equation*}\label{eq:keysign}
\frac{d}{dx}\log\frac{\Psi(ab_p(x))}{(ab_p(x))^{p-1}}
=\frac1{\Gamma_p(x)}-\frac1{\Gamma_q(x)}>0,
\end{equation*}
where Lemma \ref{lem:profiles} and the definition of $\Gamma_{r}$ give the last inequality. Since $z'(x)<0$, the ratio is strictly decreasing as a function of $z$.
\end{proof}

\section{The eigenfunction and its critical set}
\label{sec:eigen}
We recall some properties of the eigenfunctions on a convex domain. A positive first eigenfunction $u$ belongs to $C^{1,\alpha}(\Omega)$ \cite{Tol84}. On compact subsets of $\Omega\setminus\{\nabla u=0\}$ the equation is uniformly elliptic, so $u$ is smooth there and the equation holds pointwise. On a bounded convex domain, boundary regularity gives $u\in C(\overline\Omega)$ with $u=0$ on $\partial\Omega$.

\begin{lemma}\label{lem:concavity}
Let $\Omega$ be a bounded convex open set and let $u>0$ be its first $p$-eigenfunction. Then $\log u$ is concave in $\Omega$.
\end{lemma}
\begin{proof}
For convex domains with smooth boundary this is Sakaguchi's theorem \cite{Sak87}; see also \cite[Theorem 1.6]{CH26}, where general bounded convex domains are considered. For the reader's convenience, we sketch the extension, since no smoothness of $\partial\Omega$ is assumed here.

Choose smooth convex domains $D_j\Subset\Omega$ such that every compact subset of $\Omega$ is contained in $D_j$ for all sufficiently large $j$. Let $u_j$ be their positive first eigenfunctions with $\|u_j\|_{L^p(D_{j})}=1$, extended by zero to $\Omega$. Domain monotonicity gives $\lambda_p(D_{j})\ge\lambda_p(\Omega)$. Testing with any fixed compactly supported smooth function in $\Omega$, and using density in $W^{1,p}_0(\Omega)$, gives  $\limsup_{j} \lambda_{p} (D_{j}) \le \lambda_{p}(\Omega)$. Thus
\[
\lambda_p(D_{j})\longrightarrow\lambda_p(\Omega).
\]
The zero extensions are bounded in $W^{1,p}_0(\Omega)$. Weak compactness, compactness in $L^p$, and lower semicontinuity show that every subsequential limit is a normalized nonnegative minimizer in \eqref{eq:rayleigh}. Positivity and simplicity \cite{Lin90} identify it with the normalized $u$. Equality of the limiting energies also gives strong convergence in $W^{1,p}_0(\Omega)$.

Interior bounds for the eigenvalue equation and local H\"older estimates give local uniform convergence, after passage to a subsequence, on each compact subset of $\Omega$. The bounds are uniform because the eigenvalues and the $L^p$ norms are bounded and the compact subsets eventually lie in all $D_{j}$. Sakaguchi's theorem implies
\[
u_j((1-t)x+ty)\ge u_j(x)^{1-t}u_j(y)^t
\]
for $x,y\in D_{j}$ and $0<t<1$. Passing to the limit for fixed $x,y\in\Omega$ proves the assertion.
\end{proof}

Now set $v=\log u$ and $K=\{x\in\Omega:\nabla u(x)=0\}$. In $\Omega\setminus K$ we write
\begin{equation*}\label{eq:notation}
z=|\nabla v|=\frac{|\nabla u|}{u},\qquad
\nu=-\frac{\nabla u}{|\nabla u|}=-\frac{\nabla v}{|\nabla v|}.
\end{equation*}
For $0<t<\max u$ let $\Sigma_t=\{u=t\}$. At a point $x\in\Omega\setminus K$, $\nu$ is the unit normal to $\Sigma_{u(x)}$ pointing outside $\{u>u(x)\}$.
We define $H$ by
\begin{equation}\label{eq:H}
(N-1)H=\diver\nu=-\frac{\Delta u-\langle D^2u\,\nu,\nu\rangle}{|\nabla u|}
\qquad\text{in }\Omega\setminus K.
\end{equation}
Thus $H(x)$ is the mean curvature of $\Sigma_{u(x)}$ at $x$.

 It is well known that, for $w=u$ and $w=v$,
\begin{equation}\label{eq:plapnu}
\Delta_p w=|w_\nu|^{p-2}\big[(p-1)w_{\nu\nu}+(N-1)H\,w_\nu\big]\qquad\text{in }\Omega\setminus K,
\end{equation}
where $w_\nu=\nabla w\cdot\nu=-|\nabla w|$ and $w_{\nu\nu}=\langle D^2w\,\nu,\nu\rangle$.

The following lemma collects the consequences of log-concavity used
below: the critical set of $u$ is negligible, the level sets below the maximum are smooth and convex, and their mean curvature is nonnegative and not identically zero.

\medskip

\begin{lemma}\label{lem:geometry}
Let $\Omega$ be a bounded convex open set and $u$ a positive first $p$-eigenfunction. Then:
\begin{enumerate}
\item $K=\{u=\max u\}$ is a compact convex subset of $\Omega$ with empty interior, and $|K|=0$;
\item for each $0<t<\max u$, the set $\Sigma_t$ is a compact smooth hypersurface bounding the convex set $\{u>t\}$;
\item $H\ge0$ in $\Omega\setminus K$, and $H>0$ on some nonempty open subset of $\Omega\setminus K$.
\end{enumerate}
\end{lemma}
\begin{proof}
Let $u$ be a first positive eigenfunction.

\smallskip\noindent\textit{(1).} The critical points of the differentiable concave function $v=\log u$ (Lemma \ref{lem:concavity}) are exactly its maximum points. Since $\nabla u=u\nabla v$, we have $K=\{u=\max u\}$. This is convex and compactly contained in $\Omega$, by continuity and the zero boundary values. It cannot contain a ball: testing $-\Delta_pu=\lambda_pu^{p-1}$ with a nonzero nonnegative smooth function supported in that ball would give zero on the left and a positive value on the right. A convex set with empty interior lies in a proper affine subspace, so it has zero Lebesgue measure.

\smallskip\noindent\textit{(2).} Every level below the maximum avoids $K$. Its smoothness follows from local elliptic regularity and the
implicit function theorem; compactness and convexity follow from the boundary values and log-concavity.
In particular, $\Sigma_t$ bounds the smooth convex set $\{u>t\}$.

\smallskip\noindent\textit{(3).} Since $\Sigma_t$ is the boundary of the smooth convex set $\{u>t\}$, its mean curvature $H$ is nonnegative.

For strictness, fix $0<s<t<\max u$ and set $A=\{s<u<t\}$. Its closure is a compact subset of
$\Omega\setminus K$, and $\partial A=\Sigma_s\cup\Sigma_t$, with outer unit normal $\nu$ on $\Sigma_s$
and $-\nu$ on $\Sigma_t$. By \eqref{eq:H} and the divergence theorem,
\begin{equation*}
(N-1)\int_A H\dd=\int_A\diver\nu\dd 
=P(\{u>s\})-P(\{u>t\}).
\end{equation*}
The sets $\{u>t\}$ and $\{u>s\}$ are convex, and $\overline{\{u>t\}}\subset\{u>s\}$; hence
$\{u>t\}+B_\varepsilon\subset\{u>s\}$ for some $\varepsilon>0$. By the monotonicity of the perimeter
under inclusion of convex sets and the Steiner formula \cite{Sch14},
\[
P(\{u>s\})\ge P(\{u>t\}+B_\varepsilon)>P(\{u>t\}).
\]
Therefore $H$, which is continuous and nonnegative in $\Omega\setminus K$, does not vanish identically
in $A$, and it is positive on a nonempty open subset of $A$.
\end{proof}

\section{A vector field lower bound}
\label{sec:field}
The following version of the classical vector field argument retains the remainder and permits a nonsmooth critical set, which is removed by integrating over sublevel sets of $u$. For $q=2$ it reduces to the lower bound $\lambda_2\ge\inf_{\Omega}(\diver V-|V|^2)$, which, as remarked in \cite[\S5]{H60}, was essentially known to Protter in 1958, and which Hersch turned into the maximum principle $\lambda_2(\Omega)=\max_V\inf_\Omega(\diver V-|V|^2)$, starting from a method of Payne and Weinberger (see \cite[\S4,\S5]{H60} and the references therein contained). For general $q$ it is closely related to the Picone identity \cite{AH98}.

In Section~\ref{sec:proof}, Lemma~\ref{lem:field} below will be applied with
$\eta=u$, the first $p$-eigenfunction, so that $K$ is its critical set.

\medskip

\begin{lemma}\label{lem:field}
Let $\Omega$ be a bounded open set, $q>1$, and let $\eta\in C^1(\Omega)$, with
$M=\sup_\Omega \eta$ and $K=\{\eta=M\}$ (possibly empty). Assume that $|K|=0$ and that
$\nabla \eta\neq0$ in $\Omega\setminus K$. Let $V\in C^1(\Omega\setminus K;\R^N)$ satisfy
\[
V\cdot\nabla \eta\le0\quad\text{and}\quad
\diver V-(q-1)|V|^{q'}=\mu+\mathcal R\quad\text{in }\Omega\setminus K,
\qquad \mathcal R\ge0,
\]
where $q'=q/(q-1)$. Then, since $|K|=0$, $\mathcal R$ is defined a.e. in $\Omega$, and
every $w\in W^{1,q}_0(\Omega)$ satisfies
\begin{equation}\label{eq:fieldbound}
\int_\Omega|\nabla w|^q\dd\ge\mu\int_\Omega|w|^q\dd
+\int_{\Omega}\mathcal R|w|^q\dd.
\end{equation}
In particular, $\lambda_q(\Omega)\ge\mu$.
\end{lemma}
\begin{proof}
First take $w\in C_0^\infty(\Omega)$ and $t<M$, and let $\Sigma_t=\{\eta=t\}$. Since $\nabla \eta\neq0$
on $\Sigma_t$, the open set $\Omega'_{t}=\{\eta<t\}$ has, inside $\Omega$, the $C^1$ boundary $\Sigma_t$, with outer
unit normal $\nabla \eta/|\nabla \eta|$. The field $|w|^qV$ is $C^1$ in a neighbourhood of
$\overline{\Omega'_t}\cap\supp w$, which is a compact subset of $\Omega\setminus K$. The divergence
theorem and $V\cdot\nabla \eta\le0$ give
\[
\int_{\Omega'_{t}}\diver\big(|w|^qV\big)\dd=\int_{\Sigma_t}|w|^q\,V\cdot\frac{\nabla \eta}{|\nabla \eta|}\,d\mathcal H^{N-1}\le0 .
\]
Hence, by Young's inequality,
\begin{align*}
\int_{\Omega'_{t}}|w|^q\diver V\dd
&\le-q\int_{\Omega'_{t}}|w|^{q-2}w\,\nabla w\cdot V\dd\\
&\le\int_{\Omega'_{t}}|\nabla w|^q\dd+(q-1)\int_{\Omega'_{t}}|w|^q|V|^{q'}\dd,
\end{align*}
that is,
\[
\mu\int_{\Omega'_{t}}|w|^q\dd+\int_{\Omega'_{t}}\mathcal R|w|^q\dd\le\int_\Omega|\nabla w|^q\dd .
\]
As $t\uparrow M$, the sets $\Omega'_t$ increase to $\Omega'_{M}=\Omega\setminus K$. Since $\mathcal R\ge0$
and $|K|=0$, monotone convergence proves \eqref{eq:fieldbound} for smooth $w$.

Finally, approximate any $w\in W^{1,q}_0(\Omega)$ strongly by $w_j\in C_0^\infty(\Omega)$, choosing a subsequence converging almost everywhere. Fatou's lemma for the remainder and convergence of the other two integrals prove \eqref{eq:fieldbound}.
\end{proof}

\section{Proof of the main inequality}
\label{sec:proof}
Now we are in position to prove \eqref{eq:main}.

\medskip

\begin{proposition}\label{prop:mainin}
Let $\Omega\subset\R^N$, $N\ge2$, be a bounded convex open set and let $1<q<p<\infty$. Then
\[
\lambda_q(\Omega)>\left(\frac{\pi_q}{\pi_p}\right)^q\lambda_p(\Omega)^{q/p},
\]
that is, \eqref{eq:main} holds.
\end{proposition}
\begin{proof}
We divide the proof into three steps.

\medskip

\noindent\textbf{Step 1: the equation for $v=\log u$.} 
We start from a positive Dirichlet eigenfunction $u$ of $\lambda_{p}=\lambda_{p}(\Omega)$. Since $u\in C^1(\Omega)$ and $u>0$, we have $v\in C^1(\Omega)$ and
$|\nabla v|^{p-2}\nabla v=u^{1-p}|\nabla u|^{p-2}\nabla u$. Testing the equation of $u$ with $\varphi u^{1-p}$, $\varphi\in C^\infty_c(\Omega)$, which is admissible since $u$ is bounded away from zero on $\supp\varphi$, we get
\begin{equation}
\label{eq:logpde}
-\Delta_pv=(p-1)|\nabla v|^p+\lambda_{p}
\end{equation}
weakly in $\Omega$. Moreover, since $u$ is smooth in $\Omega\setminus K$, \eqref{eq:logpde} holds pointwise there.

\medskip
With the notation of Section~\ref{sec:eigen}, we have $v_\nu=-z$ and $v_{\nu\nu}=-z_\nu$, since
$\nabla z=D^2v\,\nabla v/|\nabla v|=-D^2v\,\nu$. Inserting \eqref{eq:plapnu} with $w=v$ in
\eqref{eq:logpde}, we obtain
\begin{equation}\label{eq:split}
z_\nu=\frac{z^{2-p}\big(\lambda_p+(p-1)z^p\big)}{p-1}-\frac{(N-1)H}{p-1}\,z
\qquad\text{in }\Omega\setminus K.
\end{equation}

\medskip
\noindent\textbf{Step 2: the vector field.} Let $\Psi$ be the function given by Lemma \ref{lem:transfer}
with $\lambda=\lambda_{p}=\lambda_p(\Omega)$, so that $\mu=(\pi_q/\pi_p)^q\lambda_p^{q/p}$, and set
\begin{equation*}\label{eq:V}
V=\Psi(z)\,\nu\qquad\text{in }\Omega\setminus K.
\end{equation*}
The field $V$ is smooth in $\Omega\setminus K$, $|V|=\Psi(z)$ and, since $\nu=-\nabla u/|\nabla u|$,
\begin{equation}\label{eq:outward}
V\cdot\nabla u=-\Psi(z)\,|\nabla u|<0\qquad\text{in }\Omega\setminus K,
\end{equation}
that is, $V$ points outside the superlevel sets of $u$.

\medskip
Now, for every $f\in C^1(0,\infty)$,
\[
\diver\big(f(z)\,\nu\big)=f'(z)\,z_\nu+(N-1)H\,f(z)\qquad\text{in }\Omega\setminus K.
\]
Choosing $f=\Psi$ and using \eqref{eq:split}, we get
\[
\diver V=\frac{z^{2-p}\big(\lambda_p+(p-1)z^p\big)}{p-1}\,\Psi'(z)
+(N-1)H\left(\Psi(z)-\frac{z\,\Psi'(z)}{p-1}\right).
\]

\medskip
Since $|V|^{q'}=\Psi(z)^{q'}$,
\begin{multline*}
\diver V-(q-1)|V|^{q'}
=\left[\frac{z^{2-p}(\lambda_p+(p-1)z^p)}{p-1}\,\Psi'(z)-(q-1)\Psi(z)^{q'}\right] \\
+(N-1)H\left(\Psi(z)-\frac{z\,\Psi'(z)}{p-1}\right),
\end{multline*}
and the term in square brackets is equal to $\mu$, by \eqref{eq:ode}. Moreover,
\[
\Psi(z)-\frac{z\,\Psi'(z)}{p-1}=-\frac{z^p}{p-1}\left(\frac{\Psi(z)}{z^{p-1}}\right)'.
\]
Hence
\begin{equation}\label{eq:remainder}
\diver V-(q-1)|V|^{q'}=\mu+\mathcal R,\qquad
\mathcal R=\frac{N-1}{p-1}\,H\,z^p\frac{\partial}{\partial z}\left[-\frac{\Psi(z)}{z^{p-1}}\right],
\end{equation}
where $\mu=(\pi_q/\pi_p)^q\lambda_{p}^{q/p}$.

\medskip
The remainder $\mathcal R$ is the product of three factors:
\begin{itemize}
\item $(N-1)z^p/(p-1)>0$ in $\Omega\setminus K$;
\item $-\big(\Psi(z)/z^{p-1}\big)'>0$, by \eqref{eq:sign};
\item $H\ge0$ in $\Omega\setminus K$, and $H>0$ on a nonempty open set $U\subset\Omega\setminus K$ (Lemma \ref{lem:geometry}).
\end{itemize}
Therefore $\mathcal R\ge0$ in $\Omega\setminus K$, and $\mathcal R>0$ in $U$.

\noindent\textbf{Step 3: conclusion.} By Lemma \ref{lem:geometry}, $K=\{u=\max u\}$, $|K|=0$, and $u$ is
smooth with $\nabla u\neq0$ in $\Omega\setminus K$. 
Together with \eqref{eq:outward} and \eqref{eq:remainder}, this shows that Lemma~\ref{lem:field} applies with $\eta=u$ and the field $V$.

 Let $w$ be the positive first
$q$-eigenfunction with $\int_\Omega w^q\dd=1$, so that $\int_\Omega|\nabla w|^q\dd=\lambda_q(\Omega)$. Then \eqref{eq:fieldbound} gives
\begin{equation}\label{eq:deficit}
\lambda_q(\Omega)-\left(\frac{\pi_q}{\pi_p}\right)^q\lambda_p(\Omega)^{q/p}
\ge\int_{\Omega}\mathcal R\,w^q\dd>0,
\end{equation}
where the last inequality holds because $w>0$ in $\Omega$ and $\mathcal R>0$ in the open set $U$. This concludes the proof of \eqref{eq:main}.
\end{proof}

\section{Sharpness and consequences}\label{sec:sharp}

\subsection*{Slabs and the endpoint inequalities}

The next proposition is the second ingredient in the proof of Theorem~\ref{thm:main}, which we complete below.

\begin{proposition}\label{prop:slabs}
Let $\Omega_\ell=(0,1)\times(-\ell/2,\ell/2)^{N-1}$. For every $1<r<\infty$,
\[
\lambda_r(\Omega_\ell)\longrightarrow\pi_r^r\quad(\ell\to\infty).
\]
Moreover, $\rho(\Omega_\ell)=1/2$ for $\ell\ge1$, and $h(\Omega_\ell)\to2$.
\end{proposition}
These facts are well known. The convergence of the eigenvalues follows by applying $N-1$ times the two-sided estimate in \cite[Lemma 2.4]{BBP23}; see also \cite{Br18,Br20} for the eigenvalue of the infinite slab $(0,1)\times\R^{N-1}$. Moreover, all the assertions can be deduced, in the more general anisotropic setting, from \cite[Lemma 3.1 and Theorem 5.8]{DPdBG20}. We include a short proof, since its arguments are used again in Section \ref{sec:further}.
\begin{proof}
Apply the one-dimensional Poincar\'e inequality on each segment parallel to the first coordinate axis. Since $|\nabla\varphi|\ge|\partial_1\varphi|$, integration gives $\lambda_r(\Omega_\ell)\ge\pi_r^r$.

For the opposite inequality, fix $0\ne\phi\in C_0^\infty(Q')$, where $Q'=(-1/2,1/2)^{N-1}$, and test with $\varphi_\ell(x)=y_r(x_1)\phi(x'/\ell)$. Changing variables gives the quotient
\[
\frac{\displaystyle\int_0^1\int_{Q'}
\left(y_r'^2\phi^2+\ell^{-2}y_r^2|\nabla\phi|^2\right)^{r/2}\,dx'\,dx_1}
{\displaystyle\int_0^1y_r^r\,dx_1\int_{Q'}|\phi|^r\,dx'}.
\]
Passing to the limit $\ell\to +\infty$, dominated convergence sends it to $\int_0^1|y_r'|^r/\int_0^1y_r^r=\pi_r^r$.

For $\ell\ge1$ the inradius is obviously $\rho(\Omega_\ell)=1/2$. By the first part of the proof, $\lambda_r(\Omega_\ell)\ge\pi_r^r$, that is $\Lambda_r(\Omega_\ell)\ge1$, for every $r>1$; letting $r\downarrow1$ in \eqref{eq:limits} gives $h(\Omega_\ell)\ge2$. Taking $E=\Omega_\ell$ in the definition of $h$, we conclude that
\[
2\le h(\Omega_\ell)\le\frac{P(\Omega_\ell)}{|\Omega_\ell|}
=2+\frac{2(N-1)}\ell.
\]
\end{proof}
We can finally conclude the proof of the main result
\begin{proof}[Proof of Theorem \ref{thm:main}]
Let $1<q<p<\infty$. Inequality \eqref{eq:main} is proved in
Section~\ref{sec:proof}, Proposition \ref{prop:mainin}. Finally, each
$\Omega_\ell$ is a bounded convex open set, and
Proposition~\ref{prop:slabs} gives $\lambda_r(\Omega_\ell)^{1/r}\to\pi_r$
for $r=p,q$. Hence $\F_{p,q}(\Omega_\ell)\to\pi_p/\pi_q$, and the constant
in \eqref{eq:main} cannot be lowered.
\end{proof}

\begin{remark}\label{rem:endpoints}
With the conventions \eqref{eq:conventions}, Theorem \ref{thm:main} extends to $1\le q<p\le\infty$. Indeed, let $\Omega$ be a bounded convex open set. By the strict monotonicity of $p\mapsto\Lambda_p(\Omega)$ and by \eqref{eq:limits}, for $1<s<r<t<\infty$ we have
\[
\frac{h(\Omega)}2\ge\Lambda_s(\Omega)>\Lambda_r(\Omega)>\Lambda_t(\Omega)\ge\frac1{2\rho(\Omega)},
\]
hence
\begin{equation}\label{eq:endpoints}
\frac1{2\rho(\Omega)}<\Lambda_r(\Omega)<\frac{h(\Omega)}2\qquad\text{for every }1<r<\infty,
\end{equation}
and in particular $h(\Omega)\rho(\Omega)>1$. Since the limits \eqref{eq:limits} do not rely on \eqref{eq:hersch} and \eqref{eq:buser}, this gives a new proof of the Hersch--Protter and reverse Cheeger inequalities. By Proposition \ref{prop:slabs}, all these inequalities are sharp.
\end{remark}

\subsection*{Other comparisons}
Since an inner ball is an admissible competitor for the Cheeger problem,
$h(\Omega)\rho(\Omega)\le N$. Combining \eqref{eq:lind}, \eqref{eq:endpoints} and Theorem \ref{thm:main} gives
\[
   \frac{\pi_p}{\pi_q\,h(\Omega)\rho(\Omega)}<\F_{p,q}(\Omega)<\frac{\pi_p}{\pi_q},
   \qquad \F_{p,q}(\Omega)\ge\frac qp .\]
The first inequality was already observed, in a non-strict form, in the proof of \cite[Proposition 3.2]{BBP23}.
Theorem \ref{thm:main} compares $\lambda_p^{1/p}$ and $\lambda_q^{1/q}$, which have the same scaling. The eigenvalues themselves scale differently, $\lambda_r(t\Omega)=t^{-r}\lambda_r(\Omega)$, so their order depends on the size of $\Omega$. The next corollary gives a criterion for $\lambda_p(\Omega)<\lambda_q(\Omega)$ in terms of the Cheeger constant.

\begin{corollary}\label{cor:large}
Let $\Omega\subset\R^N$ be a bounded convex open set, and let $1<q<p<\infty$. If
\begin{equation}\label{eq:hcond}
h(\Omega)\le2\left(\frac{\pi_q^q}{\pi_p^p}\right)^{1/(p-q)},
\end{equation}
then $\lambda_p(\Omega)<\lambda_q(\Omega)$. In particular, this holds whenever
\[
\rho(\Omega)\ge\frac N2\left(\frac{\pi_p^p}{\pi_q^q}\right)^{1/(p-q)}.
\]
\end{corollary}
\begin{proof}
By \eqref{eq:main}, \eqref{eq:endpoints} and \eqref{eq:hcond}, we have
\begin{multline*}
\lambda_p(\Omega)<\left(\frac{\pi_p}{\pi_q}\right)^p\lambda_q(\Omega)\,\lambda_q(\Omega)^{(p-q)/q} \\ 
\le\lambda_q(\Omega)\,\frac{\pi_p^p}{\pi_q^q}\left(\frac{h(\Omega)}2\right)^{p-q}
\le\lambda_q(\Omega).
\end{multline*}
Finally, since $h(\Omega)\rho(\Omega)\le N$, the condition on $\rho(\Omega)$ implies \eqref{eq:hcond}.
\end{proof}

\subsection*{Maximizing sequences}

\begin{corollary}\label{cor:plane}
Let $1<q<p<\infty$ and let $\tilde \Omega_j\subset\R^2$ be nonempty bounded convex open sets with $|\tilde \Omega_j|=1$. Then
\[
\F_{p,q}(\tilde \Omega_j)\longrightarrow\frac{\pi_p}{\pi_q}
\quad\Longleftrightarrow\quad\diam(\tilde \Omega_j)\longrightarrow\infty.
\]
\end{corollary}
\begin{proof}
If the diameters diverge, the result follows from \cite[Proposition 3.5]{BBP23}, since $\F_{p,q}$ is constant on intervals.

Conversely, suppose by contradiction that the sequence is maximizing and has a subsequence of bounded diameter. By translation invariance and Blaschke's selection theorem \cite{Sch14}, its closures converge in Hausdorff distance, along a further subsequence, to a compact convex set $C$. Since the volume is continuous with respect to the Hausdorff distance on convex bodies, $|C|=1$, so that $C$ has nonempty interior $\Omega$.
By the continuity of $\lambda_r$ with respect to Hausdorff convergence of convex sets, as in the proof of \cite[Theorem 3.6]{BBP23}, $\F_{p,q}(\Omega)=\pi_p/\pi_q$, contradicting Theorem \ref{thm:main}.
\end{proof}
For $p=2$ and $q=1$, the implication ``$\Leftarrow$'' was proved in \cite[Proposition 4.1]{Par17}.

\begin{remark}\label{rem:cylinders}
For $N\ge3$ the analogue of the implication ``$\Leftarrow$'' of Corollary \ref{cor:plane} fails: divergence of the diameter is not sufficient for a sequence of unit-volume convex sets to be maximizing. Let $\omega\subset\R^{N-1}$ be a bounded convex open set, let $C_\ell=\omega\times(0,\ell)$, and let $\theta_\ell=|C_\ell|^{-1/N}=(|\omega|\,\ell)^{-1/N}$, where $|\omega|$ is the $(N-1)$-dimensional measure of $\omega$. The set
\[
D_\ell=\theta_\ell\,C_\ell=(\theta_\ell\,\omega)\times(0,\theta_\ell\,\ell)
\]
has unit volume. As $\ell\to\infty$, $\lambda_r(C_\ell)\to\lambda_r(\omega)$ for $r=p,q$ by \cite[Lemma 2.4]{BBP23}. Since $\F_{p,q}$ is scale invariant, we obtain
\[
\F_{p,q}(D_\ell)=\F_{p,q}(C_\ell)\longrightarrow\F_{p,q}(\omega)<\frac{\pi_p}{\pi_q},
\]
where the strict inequality is Theorem \ref{thm:main} in dimension $N-1\ge2$. Moreover, $D_\ell$ is a thin cylinder: its base $\theta_\ell\,\omega$ has diameter of order $\ell^{-1/N}$, while its height $\theta_\ell\,\ell=\ell^{1-1/N}|\omega|^{-1/N}$ tends to infinity. Hence $\diam(D_\ell)\ge\theta_\ell\,\ell\to\infty$, so that  $(D_\ell)$ consists of unit-volume convex sets with diverging diameters, but it is not a maximizing sequence. For $p=2$ and $q=1$, with $\omega$ a ball, this example is in \cite[Example 3.12]{CLS26}.
\end{remark}

\section{Further questions}
\label{sec:further}
The remainder in \eqref{eq:deficit} gives a possible starting point for a quantitative comparison. It remains to relate its integral to geometric quantities measuring degeneration towards a slab; the present proof does not provide such an estimate.
\smallskip

The characterization of maximizing sequences in higher dimension remains open. For $p=2$ and $q=1$, Cont, Leonardi and Saracco proved that vanishing relative width, that is, vanishing ratio between the first two principal widths, is a sufficient condition in $\R^3$ and for rhomboid-like sets, and conjectured that it is also necessary \cite{CLS26}. Their proof relies on the fact that, in these two settings, if $(\Omega_n)$ is a sequence of convex bodies with vanishing relative width, then, up to isometries, $\Omega_n$ contains a box $(0,a_n)\times(0,b_n)^{N-1}$ with $a_n/w_n\to1$ and $b_n/w_n\to\infty$, where $w_n$ is the minimal width of $\Omega_n$ \cite[Theorems 4.9 and 5.5]{CLS26}; combined with the arguments of Proposition \ref{prop:slabs}, it shows that the same sufficient condition holds for every $\F_{p,q}$, $1\le q<p\le\infty$. It would be interesting to know whether, in general, the maximizing sequences are the same for all the functionals $\F_{p,q}$.

   Finally, an extension beyond convex domains would require level sets of the first $p$-eigenfunction with nonnegative mean curvature and no critical points below the maximum level. It is not known to us whether nonnegative mean curvature of $\partial\Omega$ alone supplies either property.

 \subsection*{Acknowledgments}
The author is a member of the Gruppo Nazionale per l'Analisi Matematica, la Probabilit\`a e le loro Applicazioni (GNAMPA) of the Istituto Nazionale di Alta Matematica (INdAM).

\subsection*{AI Usage Declaration}
The author conceived and developed all the results and proofs in this paper. Claude (Anthropic) was used as an AI assistant to discuss proof strategies, assist in checking technical steps and help draft the text. The author retains full responsibility for the correctness and content of the final manuscript.


\end{document}